\documentclass[12pt]{amsart}

\usepackage{enumerate, amsmath, amsthm, amsfonts, amssymb, xy,  mathrsfs, graphicx, paralist, fancyvrb}
\usepackage[usenames, dvipsnames]{xcolor}
\usepackage[margin=1in]{geometry} 
\usepackage[bookmarks, colorlinks=true, linkcolor=blue, citecolor=blue, urlcolor=blue]{hyperref}

\numberwithin{equation}{section}

\newtheorem{lemma}[equation]{Lemma}

\newtheorem{theoremst}[equation]{Theorem$^{\star}$}
\newtheorem{propositionst}[equation]{Proposition$^{\star}$}
\newtheorem{corollaryst}[equation]{Corollary$^{\star}$}
\newtheorem{lemmast}[equation]{Lemma$^{\star}$}

\theoremstyle{definition}
\newtheorem{rmk}[equation]{Remark}

\newtheorem{eg}[equation]{Example}

\newtheorem{defn}[equation]{Definition}

\newcommand{\bP}{\mathbf{P}}
\newcommand{\bS}{\mathbf{S}}
\newcommand{\bV}{\mathbf{V}}
\newcommand{\bk}{\mathbf{k}}
\newcommand{\cI}{\mathcal{I}}

\newcommand{\arxiv}[1]{{\tt \href{http://arxiv.org/abs/#1}{{arXiv:#1}}}}

\DeclareMathOperator{\Sym}{Sym}

\DeclareMathOperator{\sgn}{sgn}

\newcommand{\GL}{\mathbf{GL}}
\newcommand{\SL}{\mathbf{SL}}

\usepackage{ytableau}

\title[Equations for the fifth secant variety of Segre products]{Equations for the fifth secant variety\\ of Segre products of projective spaces}
\author{Luke Oeding} 

\address{Department of Mathematics, Auburn University, Auburn, AL}
\email{\href{mailto:oeding@auburn.edu}{oeding@auburn.edu}}
\urladdr{\url{http://www.auburn.edu/~lao0004/}}

\author{Steven V Sam}
\address{Department of Mathematics, University of California, Berkeley, CA}
\email{\href{mailto:svs@math.berkeley.edu}{svs@math.berkeley.edu}}
\urladdr{\url{http://math.berkeley.edu/~svs/}}


\date{March 25, 2015}

\begin{document}

\maketitle

\vspace{-.4in}

\begin{abstract}
We describe a computational proof that the fifth secant variety of the Segre product of five copies of the projective line is a codimension $2$ complete intersection of equations of degree $6$ and $16$. Our computations rely on pseudo-randomness, and numerical accuracy, so parts of our proof are only valid ``with high probability''. 
\end{abstract}

\section{Introduction}
Secant varieties have received growing attention in recent times, largely because of the fact that they provide a geometric model relevant to a wide variety of applications.  The purpose of this note is to provide a case study in finding equations of secant varieties. For an introduction to secant varieties and their applications we invite the reader to consult \cite{CGO} and the vast collection of references therein. The 5 factor binary secant variety is particularly interesting in light of recent results of Bocci and Chiantini \cite{BC}, that $2 \times 2 \times 2 \times 2 \times 2$ tensors are not identifiable in rank 5, but the generic tensor of that format has exactly 2 decompositions.   For $\geq 6$ factors, the binary Segre product is known to be $k$-identifiable for most of the possible values of $k$ below the generic rank, see \cite{BCO} and \cite{COV2014}.

Fix a field $\bk$ of characteristic $0$. For $i =1,\dots, 5$, let $V_i$ be a $2$-dimensional vector space over $\bk$. Let $\bV = \bigotimes_{i=1}^5 V_i$.
Let $X$ be the $5$th secant variety of $\prod_{i=1}^5 \bP(V_i)$ inside of $\bP(\bV)$ (by the Segre embedding). The goal of this note is to prove\footnote{Our methods include  probabilistic symbolic computations and numerical computations. Though they have been carefully tested and produce completely reproducible results, they are technically only true with high probability, or up to the numerical precision of the computers we use. To indicate reliance on such computations, we designate those theorems, corollaries, and propositions with a star.}
the following statement.

\begin{theoremst} \label{thm:CI}
The affine cone of $X$ is a complete intersection of two equations: one of degree $6$, and one of degree $16$.
\end{theoremst}

We speculate that the homogeneous coordinate ring of any secant variety of any Segre product of projective spaces is Cohen--Macaulay. Theorem~\ref{thm:CI} confirms this for $X$. Using flattening and inheritance \cite[Ch.~7]{LandsbergTensorBook}, we get the following corollary:

\begin{corollaryst} \label{cor:eqns}
Suppose $n\geq 5$. Let $V_{1},\dots, V_{n}$ be vector spaces, let $f_{6}$ and $f_{16}$ denote minimal generators of $\sigma_{5}((\bP^{1})^{\times 5})$ and let $F_{6}$ and $F_{16}$, respectively, denote the natural liftings of $f_{6}$ and $f_{16}$ to $\bk[V_{1}\otimes \dots\otimes V_{n}]$. Let $X$ be the fifth secant variety of 
\[
\bP(V_1) \times \cdots \times \bP(V_n) \subset \bP(V_1 \otimes \cdots \otimes V_n).
\]
Then the linear span of the  $\GL(V_1) \times \cdots\times \GL(V_n) \rtimes \Sigma_{n}$-orbits of $F_{6}$ and $F_{16}$ are equations that vanish on $X$.
\end{corollaryst}

A \emph{flattening} of a tensor $A \in \bV$ is a matrix constructed by viewing $A$ as a linear mapping from the dual of one subset of the 5 vector spaces to the complementary subset. The basic fact is that if $A$ has rank $r$, then a flattening $F(A)$ has rank $\le r$. So when non-trivial, 
the $(r+1) \times (r+1)$ minors of flattenings provide equations of secant varieties. In our case the only possible sizes of flattenings are (up to transpose)  $2\times 16$, or $4\times 8$, with maximum ranks $2$ and $4$ respectively, so they do not provide non-trivial equations for tensors of rank $5$.

A next source for equations of secant varieties are exterior (or Koszul) flattenings (see, for instance, \cite{CEO}) or, more generally, Young flattenings \cite{LanOtt11_Equations}.  For basic background, we invite the reader to consult \cite{LandsbergTensorBook}.  The basic idea of Young flattenings is to consider cases when the tensor product of a Schur module $\bS_{\mu}\bV:= \bS_{\mu_{1}} V_{1} \otimes \bS_{\mu_2}V_{2} \otimes \bS_{\mu_{3}}V_{3} \otimes \bS_{\mu_{4}}V_{4} \otimes \bS_{\mu_{5}}V_{5}$ with $\bV$ contains another Schur module $\bS_{\nu}\bV$, i.e., we have a linear map
\[
F_{\mu, \nu}\colon \bS_{\mu}\bV \to \bS_{\nu}\bV
\]
depending linearly on $A\in \bV$. Subadditivity of matrix rank implies that if $F_{\mu,\nu}(A)$ has rank $p$ when $A$ has rank $1$ then $F_{\mu,\nu}(A)$  has rank at most $r\cdot p$ when $A$ has rank $r$. The art in this approach is to find good pairs of multi-partitions so that the dimensions of $\bS_{\mu}\bV$ and $\bS_{\nu}\bV$ are large with respect to $p$, so that the Young flattening has a chance to detect high rank tensors. In principle it is possible, but tedious, to list all possible Young flattenings and check which have a chance to provide meaningful equations.  Because $V_{i}$ are all $2$-dimensional, there are not too many choices for $\mu$ and $\nu$, however our initial tries at finding Young flattenings that give non-trivial equations for the 5th secant variety were unsuccessful, and it seems that Young flattenings do not provide equations for this secant variety.

After looking for Young flattenings unsuccessfully, we attempted a systematic search for equations via interpolation informed by representation theory. Our approach relied on computer calculations, which we explain in \S\ref{sec:eqn}. The search for the equation of degree $6$ is rigorous, but the search for the degree $16$ equation is only correct up to high probability since we only show that it vanishes on sufficiently many pseudo-random points. Our search for equations was guided by our guess that this variety, having low codimension, would be defined by just a few equations, and because of the large symmetry group, that these equations would be semi-invariants.  Our guesses are validated in \S\ref{sec:proof}, where we use these equations and one additional computer calculation on the degree of $X$ (which is also only valid up to high probability) to deduce Theorem~\ref{thm:CI}. In \S\ref{sec:complement} we provide a more detailed version of Corollary~\ref{cor:eqns} from a $\Delta$-module (in the sense of \cite{delta-mod}) point of view.

In particular, the equations that we find provide modules of equations for all other $5$th secant varieties of Segre products, both when the dimensions of the factors increase (by ``inheritance'') and when the number of factors increase (by flattening). As far as we know these equations do not come from any known construction (such as Young flattenings), so they provide interesting new classes of equations for secant varieties.

\subsection*{Acknowledgements}
We thank Bernd Sturmfels for suggesting this problem, J.M. Landsberg for helpful discussions, and Jon Hauenstein for providing {\tt Bertini} help. The software {\tt Bertini} \cite{BertiniSoftware}, {\tt Macaulay2} \cite{M2} and {\tt Maple} were helpful for this work.
Both authors acknowledge the hospitality of the Simons Institute for the Theory of Computing in Berkeley where this work was carried out.
SS was supported by a Miller research fellowship.

\section{A search for equations guided by symmetry} \label{sec:eqn}

\subsection{General idea} \label{sec:eqn-gen}
The variety $X$ that we are studying has low codimension so we expect its ideal to be cut out by few equations. In addition, the defining ideal of $X$ has a large symmetry group. If one polynomial is in the ideal of $X$, then so is the entire vector space of polynomials in the span of its orbit. So we expect $X$ to be cut out by a small number of semi-invariant polynomials.
In this section we describe how we use all available symmetry to cut down our search for equations.

Choose a basis $e_0, e_1$ for $V_i$ so that we can identify the coordinates of $\bP(\bV)$ with $x_I$ where $I \in \{0,1\}^5$. Let $R$ denote the polynomial ring $\Sym(\bV) \cong \bk[x_{I} \mid I \in \{0,1\}^{5} ]$ and let $\mathcal I:=\mathcal I(X)$ be the ideal of equations vanishing on $X$. Since $R$ is graded, $R = \bigoplus_{d}R_{d}$, and we can compute $\mathcal{I}_{d}\subset R_{d}$ for each $d$.  The most naive approach to determining $\mathcal I_{d}$ is to evaluate a basis of $R_{d}$ on $\dim R_{d}$ points of $X$ using the parametrization of $X$, store the results in a matrix $M$ and compute the kernel $\ker(M)= \mathcal{I}_{d}$. In practice we use pseudo-random points on $X$ with rational coefficients. 

In exact arithmetic, non-vanishing is a certainty, but vanishing might yield a false-positive. 
So this test gives an upper bound for $\dim \mathcal{I}_{d}$ and a probabilistic lower bound.  The confidence in the lower bound may be increased by evaluating on more points of $X$. Alternatively, one can work over a function field over $\bk$ of large enough transcendence degree (i.e., use parametrized points) where vanishing yields a genuine equation. The downside is that such computations are more expensive. This approach only works for small values of $d$ as the dimension of $R_{d}$ grows quickly. In particular, one can use it for $d=6$ (the lowest possible degree in which the equations of $\sigma_{5}$ can occur, a basic fact from the theory of prolongation \cite[Corollary 3.4]{LM04}) to discover the equation $f_6$ in \S\ref{app:f6}, but it will not work for $d=16$ (the largest degree we tested).

However, $R_d$ has an action of the group $\SL_2^{\times 5}$ and we suspected that $X$ is  defined by invariants of $\SL_2^{\times 5}$. This gives a much smaller space of functions to search. Set $U_d = (\Sym^d \bV)^{\SL_2^{\times 5}}$ and $T_{d} = \bV^{\otimes d}$.
For each even degree $d=2m$ the space of $\SL_2^{\times 5}$-invariants in $T_{d}$ has a basis consisting of quintuples of Young tableaux each of shape $(m,m)$ (there are no invariants in odd degree). A quintuple of tableau can be interpreted as a function on $\bV$ by applying the Young symmetrizer associated to the quintuple of tableaux. Since $R_{d}$ is a quotient of $T_d$, it is spanned by linear combinations of quintuples of Young tableaux which now satisfy certain linear dependencies.

An explicit basis of $U_{d} \subset R_d$ may be found without explicit knowledge of all of the relations as follows. We can verify that a given set of quintuples of tableaux are linearly independent by evaluating them on dimension-many pseudo-random points of $\bV$. If the matrix constructed in this way has full rank, then we have a basis of that space of invariants. If not, we continue selecting random quintuples until a basis is found. 

Finally, the space $U_d$ has an additional action of $\Sigma_5$. 
Assuming that there is a single minimal generator of $\mathcal{I}(X)$ in a given $U_d$, it must be a semi-invariant of $\Sigma_5$, so either an invariant or skew-invariant. Let  $U_d^{\Sigma_5}$ (respectively $U_d^{\Sigma_5, {\rm sgn}}$) denote the subspace of $U_{d}$ of $\Sigma_{5}$-invariants (respectively skew-invariants).
In Figure~\ref{Ud} we list the dimensions of these spaces of invariants for degrees up to $16$. The results follow from standard character theory calculations whose explanation we will omit.
\begin{figure}\label{Ud}
\begin{center}
\begin{tabular}{l|l|l|l}
Degree $d$ & $\dim U_d$ & $\dim U_d^{\Sigma_5}$ & $\dim U_d^{\Sigma_5, {\rm sgn}}$ \\ \hline
2 & 0 & 0 & 0 \\
4 & 5 & 1 & 0\\
6 & 1 & 0 & 1 \\
8 &36 & 4 & 0\\
10 & 15 & 0 & 2\\
12 & 228 & 12 & 2\\
14 & 231 & 2 & 9\\ 
16 & 1313 & 39 & 10
\end{tabular}
\end{center}
\caption{The dimensions of $U_d$, and its subspaces of $\Sigma_5$-invariants and skew-invariants.
}
\end{figure}
We focus our search for equations of $X$ in the space of semi-invariants for  $\SL(2)^{\times 5}\rtimes\Sigma_5$.
We describe this procedure in the case $d=16$ in \S\ref{app:f16}. The following proposition is a summary of what we found:

\begin{propositionst} \label{prop:eqns}
There are minimal equations $f_6, f_{16}$ vanishing on $X$ of degrees $6$ and $16$. Both are invariant under $\SL_2^{\times 5}$. Furthermore, $f_6$ is a skew-invariant under $\Sigma_5$ while $f_{16}$ is a $\Sigma_5$-invariant.
\end{propositionst}

To clarify: $f_6$ was constructed explicitly and we verified symbolically that it vanishes on $X$ and that it is a skew-invariant. The polynomial $f_{16}$ was also constructed explicitly and verified to be a $\Sigma_5$-invariant, but we only verified that it vanishes on a large collection of pseudo-random points on $X$, and hence it belongs to the ideal of $X$ with high probability.

\subsection{The equation $f_6$} \label{app:f6}
Given a monomial in the $x_I$ (the coordinates on $\bP(\bV)$), define its skew-symmetrization to be $c^{-1} \sum_{\sigma \in \Sigma_5} {\rm sgn}(\sigma) x_{\sigma(I)}$ where $c$ is the coefficient of $x_I$ in the sum. The polynomial $f_6$ has $864$ monomials and  is the sum of the skew-symmetrizations of the following $15$ monomials:
\begin{align*}
-x_{00000} x_{01010} x_{01101} x_{10011} x_{10100} x_{11111} , \quad 
x_{00000} x_{01100} x_{01111} x_{10010} x_{10111} x_{11001}, \\ 
-x_{00000} x_{01100} x_{01111} x_{10011} x_{10110} x_{11001} , \quad
x_{00000} x_{01101} x_{01110} x_{10011} x_{10110} x_{11001}, \\ 
-x_{00110} x_{01000} x_{01101} x_{10000} x_{10011} x_{11111} , \quad
x_{00100} x_{01010} x_{01111} x_{10000} x_{10111} x_{11001}, \\ 
x_{00100} x_{01000} x_{01111} x_{10011} x_{10110} x_{11001} , \quad
x_{00110} x_{01000} x_{01101} x_{10001} x_{10010} x_{11111}, \\ 
-x_{00100} x_{01010} x_{01111} x_{10001} x_{10111} x_{11000} , \quad
x_{00100} x_{01010} x_{01111} x_{10011} x_{10101} x_{11000}, \\ 
-x_{00101} x_{01010} x_{01111} x_{10000} x_{10110} x_{11001} , \quad
x_{00100} x_{01011} x_{01110} x_{10011} x_{10101} x_{11000}, \\ 
-x_{00110} x_{01001} x_{01100} x_{10001} x_{10010} x_{11111} , \quad
x_{00110} x_{01001} x_{01111} x_{10011} x_{10100} x_{11000}, \\ 
x_{00111} x_{01010} x_{01101} x_{10011} x_{10100} x_{11000}.
\end{align*}

There is an alternative description in terms of Young symmetrizers, following the same construction outlined in \cite{BatesOeding}. The  Young symmetrizer algorithm  takes as input  a set of fillings of five Young diagrams, performs a series of skew-symmetrizations and symmetrizations, and produces as output a polynomial in the associated Schur module.
One can search over all possible shapes and fillings of tableaux for a fixed number of boxes and evaluate Young symmetrizers to find modules of equations in $\mathcal I (X)$. We describe this method more fully in the next section.  In degree $6$ the situation is particularly nice. It turns out that there are 5 standard tableaux of shape $(3,3)$ and content $\{1,2,\dots,6\}$ and the following Schur module, which uses one of each of the 5 standard fillings, realizes the non-trivial copy of $\bigotimes_{i=1}^{5}(\bS_{3,3}V_{i})$ inside of $\Sym^{6}(\bV)$:
\ytableausetup{smalltableaux}
\[
\bS_{\ytableaushort{135,246}} V_{1} \otimes
\bS_{\ytableaushort{134,256}} V_{2} \otimes
\bS_{\ytableaushort{125,346}} V_{3} \otimes
\bS_{\ytableaushort{124,356}} V_{4} \otimes
\bS_{\ytableaushort{123,456}} V_{5}.
\]
From this description of the invariant, one can use the classical symbolic method to verify that a general point of the fifth secant variety must be a zero of this invariant (see \cite{RaicuGSS} for complete descriptions of this type of argument, or also \cite{Ottaviani09_Waring}).  
This gives an unconditional proof that $f_{6}$ is in $\mathcal I(X)$.  Moreover, in \S\ref{sec:complement} we explain how this description of $f_{6}$ also provides a generalization to 5th secant varieties of any larger number of Segre products of projective spaces of any dimensions.

\subsection{The equation $f_{16}$} \label{app:f16}

We follow the approach outlined in \S\ref{sec:eqn-gen}. Our search for new equations in degrees $8$, $10$, $12$, and $14$ did not yield any new equations, so we only describe our process in the degree $16$ case. We used {\tt Maple}, and the code for these computations may be found in the ancillary files accompanying the arXiv version of this paper. Using the approach, we conclude that there are no $\Sigma_{5}$ skew-invariants in $\mathcal{I}_{16}$. We found that $\dim(\cI\cap U_{16}^{\Sigma_{5}}) \le 3$ and equals $3$ with high probability.
Note that for $f' \in U_{10}$, we have $f_6 f' \in U_{16}^{\Sigma_5}$ if and only if $f' \in U_{10}^{\Sigma_5, \sgn}$ since $f_6$ is a skew-invariant under $\Sigma_5$. Since $U_{10}^{\Sigma_{5},{\rm sgn}}$ is $2$-dimensional, we know that $f_{6}\cdot U_{10}^{\Sigma_{5}, \sgn}$ is a $2$-dimensional subspace of $U_{16}^{\Sigma_{5}}$. Since $\dim(\cI\cap U_{16}^{\Sigma_{5}}) = 3$ (with high probability), we find that $\cI$ has one additional minimal generator in degree $16$. 

Let us describe in more detail the case of finding $\Sigma_{5}$-invariants, the skew-invariant case is similar. Let $Q_{1},Q_{2},\dots$ denote quintuples of Young tableaux all of shape $(8,8)$. The sum
\[
F_{i}:=\sum_{\sigma\in \Sigma_{5}} \sigma.Q_{i}
\] 
is in $U_{16}^{\Sigma_{5}}$. The Young symmetrizer algorithm (see \cite{BatesOeding}) can be used to evaluate $F_{i}$ on a point of $\bV$ to test if it is non-zero. Each evaluation took between 500 and 23000 seconds and up to approximately 10GB of RAM on our servers\footnote{One with 24 cores 2.8 GHz Intel Xeon processors with 144 GB RAM and another with 40 cores 2.8 GHz Intel Xeon processors with 256GB of RAM.}.

We continued this randomized search for non-zero elements of $U_{16}^{\Sigma_{5}}$ until $39$ linearly independent invariants were found (the list of quintuples of fillings, data points, and code can be found in the ancillary files).  To verify independence, we chose $39$ pseudo-random points $v_{i} \in \bV$ and verified that the matrix $M:= \left(F_{i}(v_{j})\right)$ has full rank. While each evaluation $F_{i}(v_{j})$ is expensive, they are all independent computations. The distributed computation took approximately 2-3 days of computational time to verify the independence of $F_{i}$.

We then took $45$ pseudo-random points $p_{i}\in X$ and computed the matrix $(F_{i}(p_{j}))$  (only $39$ points are strictly necessary, but we included $6$ more to increase the probability that our result is correct).  This computation took an additional 2-3 days to complete. Finally we found that the matrix $\left(F_{i}(p_{j})\right)$ has rank $36$, so there is a $3$-dimensional space of invariants vanishing on $\{p_{j}\}_{j=1}^{45}$ and vanishing on all of $X$ with high probability.

An expression of a basis of the kernel of the transpose of $(F_{i}(p_{j}))$ gives, in turn, a basis of $\mathcal{I}_{16}\cap U_{16}^{\Sigma_{5}}$ via linear combinations of quintuples of symmetrized sums of tableaux. Modding out by the space of invariants generated by $f_{6}$, (in principle) one finds $f_{16}$.

\section{Proof of Theorem~\ref{thm:CI}} \label{sec:proof}

\begin{lemma} \label{lem:codim2}
$X$ has codimension $2$.
\end{lemma}

\begin{proof}
This follows from \cite[Theorem 4.1]{CGG}.
\end{proof}

\begin{lemmast} \label{lem:bertini}
$\deg(X) \ge 96$.
\end{lemmast}

\begin{proof} 
Computing the degree of an algebraic variety is a basic function of the software {\tt Bertini} \cite{BertiniSoftware}, see the ancillary files and \cite{BertiniBook}. We thank Jon Hauenstein for his help with this computation. 
The idea is to intersect $X$ with a randomly chosen $\bP^2$ and find $96$ points. One of the basic methods of numerical algebraic geometry (and {\tt Bertini}, in particular) is numerical homotopy continuation, which relies on high precision numerical path tracking to follow the paths traced out by a homotopy from a system of polynomial equations with known roots to the desired system.  For more details on these types of computations applied to tensor decomposition, see \cite{DHO, HOOS}.

The points found are represented by floating point numbers, so only satisfy the equations approximately. Though these methods have been rigorously tested through countless examples, are open source and repeatable, there is still a chance that the computations yield a false positive result, so we only claim that the result holds with high probability, so the proof of the lemma may be read as evidence for the statement.
\end{proof}

\begin{proof}[Proof of Theorem~\ref{thm:CI}]
Using the notation from Proposition~\ref{prop:eqns}, let $Y = V(f_6, f_{16})$, which is a complete intersection. We have $X \subseteq Y$ and $\deg X  \ge \deg Y= 96$.
 Since $X$ is irreducible of codimension $2$ (Lemma~\ref{lem:codim2}), and $Y$ is equidimensional, $Y$ is also irreducible (otherwise the degree inequality would be violated). So $X$ is the reduced subscheme of $Y$. Also, this implies that they have the same degree, so $Y$ is generically reduced. Since $Y$ is Cohen--Macaulay, generically reduced is equivalent to reduced. Hence $X=Y$ is a complete intersection.
\end{proof}

\section{Generalizations via flattening and inheritance} \label{sec:complement}

As mentioned in the introduction, the equations $f_{6}$ and $f_{16}$ provide modules of equations for all other 5th secant varieties of Segre products, both when the dimensions of the factors increase (by ``inheritance'') and when the number of factors increase (by flattening). Inheritance for secant varieties was introduced in \cite{LM04}, and was reinvestigated many times since. Inheritance can be encoded in the formality of  $\Delta$-modules \cite{delta-mod}; we hope that our description here provides an entry point to this formalism for the unfamiliar reader.

The assignment $(V_1, \dots, V_5) \mapsto \Sym(V_1 \otimes \cdots \otimes V_5)$ is a multivariate polynomial functor $R$ on $5$-tuples of vector spaces, as is the assignment of the coordinate ring of the $5$th secant variety of the corresponding Segre product. Similarly, there is a functor $T_{i,j}$ defined by 
\[
T_{i,j}(V_1, \dots, V_5) = {\rm Tor}_i^{\Sym(V_1 \otimes \cdots \otimes V_5)}(R(V_1, \dots, V_5), \bk)_j
\]
(the Tor groups are ${\bf Z}$-graded, and the subscript denotes the $j$th homogeneous piece).
When $i=1$, this is the space of minimal generators in degree $j$ of the ideal of $R(V_1, \dots, V_5)$. 

As a representation of $\GL_2$, a $\SL_2$-invariant of degree $2n$ is the Schur functor $\bS_{n,n}$. So Proposition~\ref{prop:eqns} can be interpreted as saying that $\bS_{3,3} \boxtimes \cdots \boxtimes \bS_{3,3}$ appears with multiplicity $1$ in $T_{1,6}$ (coming from $f_6$) and $\bS_{8,8} \boxtimes \cdots \boxtimes \bS_{8,8}$ appears with multiplicity $1$ in $T_{1,16}$ (coming from $f_{16}$). The Koszul relation amongst $f_6$ and $f_{16}$ also shows that $\bS_{11,11} \boxtimes \cdots \boxtimes \bS_{11,11}$ appears with multiplicity $1$ in $T_{2,22}$. Furthermore, if any other product of Schur functors $\bS_{\lambda^1} \boxtimes \cdots \boxtimes \bS_{\lambda^5}$ appears in any $T_{i,j}$ then $\ell(\lambda^k) > 2$ for some $k$ since it vanishes when we evaluate on $(\bk^2, \dots, \bk^2)$.

There is another interpretation of our results using $\Delta$-modules (see \cite{delta-mod}): the $5$th secant variety of the Segre product of projective spaces is a $\Delta$-variety, and hence the assignment of a tuple of vector spaces to the space of degree $d$ equations vanishing on the $5$th secant variety is a finitely generated $\Delta$-module (which implies finitely presented using \cite[Theorem 9.1.3]{grobner}). For $d=6$, we have shown that $\bS_{3,3} \boxtimes \cdots \boxtimes \bS_{3,3}$ ($5$ copies) are minimal generators of this $\Delta$-module, and similarly for $d=16$ and $\bS_{8,8} \boxtimes \cdots \boxtimes \bS_{8,8}$.

\end{document}